\documentclass[11pt]{amsart}
\usepackage[T1]{fontenc}
\usepackage[utf8]{inputenc}

\usepackage{amsmath,amssymb,amsthm,amsfonts,mathtools,bm}

\usepackage{lmodern}
\usepackage[final]{microtype}
\usepackage{babel}
\usepackage{nicefrac} 

\usepackage[margin=1.2in]{geometry}

\usepackage{multicol}
\usepackage{array}
\usepackage{booktabs}
\usepackage[shortlabels]{enumitem}

\usepackage{tikz}
\usetikzlibrary{cd}
\usetikzlibrary{shapes}
\usetikzlibrary{shapes.multipart}

\usepackage{thmtools}
\usepackage{thm-restate}
\usepackage[hyphens]{url}
\usepackage[colorlinks]{hyperref}
\hypersetup{
    colorlinks,
    linkcolor={red!50!black},
    citecolor={blue!50!black},
    urlcolor={blue!80!black}
}

\theoremstyle{plain}
\newtheorem{theorem}{Theorem}[section]

\newtheorem*{claim*}{Claim}

\newtheorem{proposition}[theorem]{Proposition}
\newtheorem*{proposition*}{Proposition}

\newtheorem*{fact*}{Fact}

\newtheorem*{conjecture*}{Conjecture}
\newtheorem{corollary}[theorem]{Corollary}
\newtheorem{lemma}[theorem]{Lemma}
\newtheorem*{lemma*}{Lemma}

\newtheorem{question}{Question}
\newtheorem*{question*}{Question}
\theoremstyle{definition}\newtheorem{remark}[theorem]{Remark}
\theoremstyle{definition}\newtheorem*{remark*}{Remark}
\theoremstyle{definition}\newtheorem{definition}[theorem]{Definition}
\theoremstyle{definition}\newtheorem*{definition*}{Definition}
\theoremstyle{definition}
\theoremstyle{definition}
\theoremstyle{definition}\newtheorem{example}[theorem]{Example}
\theoremstyle{definition}\newtheorem*{example*}{Example}

\newcommand{\term}{\textbf} 
\renewcommand{\models}{\vDash} 
\newcommand{\notmodels}{\nvDash} 
\newcommand{\proves}{\vdash} 
\newcommand{\0}{\varnothing} 
\renewcommand{\bar}{\overline} 
\renewcommand{\phi}{\varphi} 
\renewcommand{\epsilon}{\varepsilon} 

\DeclareMathOperator{\uh}{\upharpoonright} 
\DeclareMathOperator{\Con}{\mathsf{Con}}

\def\N{\mathbb{N}} 
\def\Z{\mathbb{Z}} 
\def\Q{\mathbb{Q}} 
\def\Cantor{2^\omega} 

\def\PA{\mathsf{PA}}
\def\RCA{\mathsf{RCA}}

\def\ZF{\mathsf{ZF}}
\def\ZFC{\mathsf{ZFC}}

\let\lit\underline
\DeclareMathOperator{\Th}{Th}
\DeclareMathOperator{\Succ}{Succ}
\DeclareMathOperator{\Pred}{Pred}
\def\L{\mathcal{L}}
\DeclareMathOperator{\rad}{rad}
\newcommand{\unprime}[1]{#1^{\L' \to \L}}
\newcommand{\toprime}[1]{#1^{\L \to \L'}}
\usepackage{stmaryrd}
\newcommand{\tv}[2]{\llbracket #1 \rrbracket^{#2}}

\begin{document}
\title{A theory satisfying a strong version of Tennenbaum's theorem}
\author{Patrick Lutz}
\author{James Walsh}
\address{Department of Mathematics, University of California, Los Angeles}
\email{pglutz@math.ucla.edu}
\address{Department of Philosophy, New York University}
\email{jmw534@nyu.edu}

\begin{abstract}
We answer a question of Pakhomov by showing that there is a consistent, c.e.\ theory $T$ such that no theory which is definitionally equivalent to $T$ has a computable model. A key tool in our proof is the model-theoretic notion of mutual algebraicity. 
\end{abstract}
\maketitle
\thispagestyle{empty}


\section{Introduction}

Tennenbaum's theorem states that there is no computable nonstandard model of $\PA$~\cite{tennenbaum1959non}. Often, this result is viewed as giving us one reason the standard model of $\PA$ is special---it is the only computable model---but another perspective is possible: Tennenbaum's theorem is a source of examples of consistent, c.e.\ theories with no computable models.

To explain this perspective, let us say that a theory $T$ has the \term{Tennenbaum property} if $T$ has no computable models. Tennenbaum's theorem implies that there are many consistent extensions of $\PA$ with the Tennenbaum property. For example, the theory $\PA + \lnot \Con(\PA)$ (which asserts that $\PA$ is inconsistent) is a consistent extension of $\PA$ with only nonstandard models and hence, by Tennenbaum's theorem, with no computable models. Furthermore, a slight extension of the proof of Tennenbaum's theorem can be used to prove that many other theories have the Tennenbaum property. For example, it is not hard to show that $\ZFC$ has no computable models~\cite{hamkins2010computable} and likewise for much weaker theories like $\mathsf{Z}_2$ (the theory of full second order arithmetic), or even $\RCA_0$ (at least if ``model'' is understood in the usual sense of first order logic). More generally, it seems to be an empirical fact that every natural theory which interprets even a small fragment of second order arithmetic has the Tennenbaum property.

Recently, however, Pakhomov showed that this phenomenon is somewhat fragile: it depends on the specific language in which the theory is presented~\cite{pakhomov2022escape}. To make this idea precise, Pakhomov used the notion of \term{definitional equivalence} (also known as \term{synonymy}), a strong form of bi-interpretability introduced by de Bouv\'ere in~\cite{debouvere1965logical}. Roughly speaking, theories $T$ and $T'$ in languages $\L$ and $\L'$ are definitionally equivalent if they can be viewed as two instances of a single theory, but with different choices of which notions to take as primitive.

\begin{theorem}[Pakhomov]
There is a theory $T$ which is definitionally equivalent to $\PA$ such that any consistent, c.e.\ extension of $T$ has a computable model.
\end{theorem}

This theorem implies that every consistent, c.e.\ extension of $\PA$ is definitionally equivalent to a theory with a computable model. Moreover, the techniques used by Pakhomov are not restricted to extensions of $\PA$. For example, Pakhomov notes that they are sufficient to prove that $\ZF$ is definitionally equivalent to a theory with a computable model. More generally, Pakhomov's techniques seem sufficient to prove that each example we have given so far of a theory with the Tennenbaum property is definitionally equivalent to a theory without the Tennenbaum property.

In light of these observations, Pakhomov asked how general this phenomenon is~\cite{pakhomov2022escape}. In particular, does it hold for every consistent, c.e.\ theory?

\begin{question}[Pakhomov]
Is every consistent, c.e.\ theory definitionally equivalent to a theory with a computable model?
\end{question}

The purpose of this paper is to answer this question in the negative. In other words, to give an example of a consistent, c.e.\ theory which satisfies a strong version of the Tennenbaum property.

\begin{theorem}
\label{thm:main}
There is a consistent, c.e.\ theory $T$ such that no theory which is definitionally equivalent to $T$ has a computable model.
\end{theorem}

To prove this theorem, we construct a theory $T$ which has no computable models but is also model-theoretically tame. A key observation in our proof is that if a theory $T$ is sufficiently tame then any theory definitionally equivalent to $T$ must also be fairly tame. In particular, if $T$ is sufficiently tame then every theory which is definitionally equivalent to $T$ satisfies a weak form of quantifier elimination. 

Here's why this is useful. Suppose that $M$ is a model of a theory $T'$ which is definitionally equivalent to $T$. It follows from the definition of ``definitionally equivalent'' that within $M$, we can define a model of $T$. If $T'$ had quantifier elimination then we could assume that this definition is quantifier free and thus $M$ can compute a model of $T$. Since $T$ has no computable models, this would imply that $M$ itself is not computable. Unfortunately, we can't quite follow this strategy: we don't know that $T'$ has full quantifier elimination, but only a weak version of it. However, using this weak form of quantifier elimination we can show that $M$ can computably approximate a model of $T$ and, by picking $T$ so that its models cannot even be computably approximated, this is enough to show that $M$ is not computable.


The specific form of model-theoretic tameness that we use in our proof is known as \term{mutual algebraicity}, first defined in~\cite{goncharov2003trivial} and subsequently developed by Laskowski and collaborators (e.g.~\cite{laskowski2013mutually, laskowski2015characterizing,braunfeld2022mutual}). The main result we need from the theory of mutual algebraicity is a quantifier elimination theorem proved by Laskowski in~\cite{laskowski2009elementary}.

Our use of tame model theory in this paper is somewhat reminiscent of techniques used by Emil Je\v{r}\'{a}bek in the paper~\cite{jerabek2020recursive}. In that paper, Je\v{r}\'{a}bek separated two conditions which imply that a theory $T$ is essentially undecideable: the condition that $T$ can represent all partially recursive functions and the condition that $T$ interprets Robinson's $R$. To accomplish this, he used the fact that the model completion of the empty theory in an arbitrary language is model-theoretically tame---in particular, it eliminates $\exists^\infty$ and is $\mathsf{NSOP}$. He ended the paper by asking whether there are more connections between formal arithmetic and tame model theory. We believe our results constitute a partial answer to his question.

\subsection*{Acknowledgements} We thank Peter Cholak, Nick Ramsey, Charlie McCoy, Andrew Marks, Forte Shinko, Mariana Vicaria and Kyle Gannon for helpful conversations, James Hanson for pointing us to the literature on mutual algebraicity and Chris Laskowski for help in understanding that literature.

\section{Preliminaries on definitional equivalence and mutual algebraicity}

In this section we will give the formal definition of definitional equivalence, fix some notation related to it and review the facts about mutual algebraicity that we need.

\subsection{Definitional equivalence}
\label{sec:de}

To define definitional equivalence, we first need the concept of a definitional extension of a theory.

\begin{definition}
Given a theory $T$ in language $\L$, a \term{definitional extension} of $T$ is a theory $T' \supseteq T$ in a language $\L' \supseteq \L$ such that
\begin{enumerate}
    \item \textbf{\boldmath$T'$ is conservative over \boldmath$T$:} for each sentence $\phi \in \L$, $T'\proves \phi$ if and only if $T \proves \phi$.
    \item \textbf{The symbols in \boldmath$\L'$ are definable in \boldmath$\L$:} for each constant symbol $c$, relation symbol $R$ and function symbol $f$ in $\L'$, there is a corresponding formula $\phi_c$, $\phi_R$, or $\phi_f$ in $\L$ such that
    \begin{align*}
        T' &\proves \forall x\, (x = c \leftrightarrow \phi_c(x))\\
        T' &\proves \forall \bar{x}\, (R(\bar{x}) \leftrightarrow \phi_R(\bar{x}))\\
        T' &\proves \forall \bar{x}, y\, (f(\bar{x}) = y \leftrightarrow \phi_f(\bar{x}, y)).
    \end{align*}
\end{enumerate}
\end{definition}

\begin{definition}
Theories $T$ and $T'$ in disjoint signatures are \term{definitionally equivalent} if there is a single theory which is a definitional extension of both $T$ and $T'$.
\end{definition}

More generally, theories $T$ and $T'$ are definitionally equivalent if they are definitionally equivalent after renaming their symbols to make their signatures disjoint. However, there is no loss of generality from ignoring theories with overlapping signatures, so we will do that for the rest of this paper.

\begin{example}
The theories of the integers with plus and with minus---i.e.\ $T = \Th(\Z, +)$ and $T' = \Th(\Z, -)$---are definitionally equivalent because plus and minus can both be defined in terms of the other. More formally, the theory $T'' = \Th(\Z, +, -)$ is a definitional extension of both $T$ and $T'$. In contrast, it is well-known that the theories $\Th(\Z, +)$ and $\Th(\Z, \times)$ are \emph{not} definitionally equivalent, because neither plus nor times can be defined in terms of the other.
\end{example}

A key point about definitional equivalence is that if $T$ and $T'$ are definitionally equivalent theories in languages $\L$ and $\L'$, respectively, then every model of $T$ can be viewed as a model of $T'$ and vice-versa. Likewise, every $\L$-formula can be viewed as an $\L'$-formula and vice-versa. It will be useful to us to make this idea precise and to fix some notation.

\medskip\noindent\textbf{Translating models.}
Suppose that $T$ and $T'$ are definitionally equivalent theories in languages $\L$ and $\L'$, respectively. Let $T''$ be an $\L''$-theory witnessing the definitional equivalence of $T$ and $T'$---i.e.\ $\L\cup \L' \subseteq \L''$ and $T''$ is a definitional extension of both $T$ and $T'$. 

Suppose that $R$ is a relation symbol in $\L$. Since $T''$ is a definitional extension of $T'$, there is an $\L'$-formula, $\phi_R$, which $T''$ proves is equivalent $R$. We will refer to this formula as the \term{\boldmath$\L'$-definition of \boldmath$R$}. Similarly, every other constant, relation and function symbol of $\L$ has an $\L'$-definition and vice-versa. 

Given a model $M$ of $T'$, we can turn $M$ into an $\L$-structure by interpreting each constant, relation and function symbol of $\L$ according to its $\L'$-definition.\footnote{Technically this requires checking that for every function symbol $f$ in $\L$, $T'$ proves that the $\L'$-definition of $f$ is a function and that for every constant symbol $c$ of $\L$, $T'$ proves that the $\L'$-definition of $c$ holds of exactly one element. These both follow from conservativity of $T''$.} Furthermore, it is not hard to check that the resulting $\L$-structure is always a model of $T$. We will denote the model produced in this way by $\unprime{M}$. Likewise, if $M$ is a model of $T$ then we can transform it into a model of $T'$, which we will denote $\toprime{M}$. 

It is important to note that for any model $M$ of $T'$, $M$ and $\unprime{M}$ have the same underlying set and $\toprime{(\unprime{M})} = M$. Thus we may think of $M$ and $\unprime{M}$ as two different ways of viewing the same structure.

\medskip\noindent\textbf{Translating formulas.}
A similar transformation is possible for formulas. Suppose $\phi$ is an $\L$-formula. Then by replacing each constant, relation and function symbol in $\phi$ by the corresponding $\L'$-definition, we obtain an $\L'$-formula, which we will denote $\toprime{\phi}$. Likewise we can transform any $\L'$-formula $\phi$ into an $\L$-formula, which we will denote $\unprime{\phi}$.

\begin{example}
Suppose $f$ is a unary relation symbol in $\L$, $\phi_f(x, y)$ is its $\L'$-definition and $\psi$ is the $\L$-formula $\forall x, y\, (f(f(x)) = f(y))$. Then $\toprime{\psi}$ is the formula $\forall x, y\, (\exists z_1, z_2, z_3\, (\phi_f(x, z_1) \land \phi_f(z_1, z_2) \land \phi_f(y, z_3) \land z_2 = z_3))$. 
\end{example}

It is not hard to check that our translations of models and of formulas are compatible with each other. In particular, if $M$ is a model of $T'$, $\phi$ is an $\L'$-formula and $\bar{a}$ is a tuple in $M$ then $M \models \phi(\bar{a})$ if and only if $\unprime{M} \models \unprime{\phi}(\bar{a})$. Note that this implies that $M$ and $\unprime{M}$ have the same algebra of definable sets.

\subsection{Mutual algebraicity}

As mentioned in the introduction, we will use the model-theoretic notion of mutual algebraicity. The key definitions are of mutually algebraic formulas and mutually algebraic structures.

\begin{definition}
Given a structure $M$, a formula $\phi(\bar{x})$ with parameters from $M$ is \term{mutually algebraic over \boldmath$M$} if there is some number $k \in \N$ such that for every nontrivial partition $\bar{x} = \bar{x}_0 \cup \bar{x}_1$ and every tuple $\bar{a}_0$ in $M$, there are at most $k$ tuples $\bar{a}_1$ such that $M \models \phi(\bar{a}_0, \bar{a}_1)$.
\end{definition}

Note that the mutual algebraicity depends on what the free variables of the formula are. In particular, it is not preserved by adding dummy variables. Also note that any formula with at most one free variable is mutually algebraic.

\begin{example}
If $M$ is the structure $(\N, +)$ then the formula $x = y + 5$ is mutually algebraic over $M$ because if we fix $x$ there is at most one $y$ satisfying the formula, and vice-versa. On the other hand, the formula $x = y + z + 5$ is not mutually algebraic over $M$ because when we fix $z$ there are infinitely many pairs $x, y$ which satisfy the formula.
\end{example}

\begin{definition}
A structure $M$ is \term{mutually algebraic} if every formula is equivalent to a Boolean combination of formulas which are mutually algebraic over $M$ (and which are allowed to have parameters from $M$).
\end{definition}

\begin{example}
The structure $(\N, \Succ)$ of natural numbers with the successor function has quantifier elimination and thus every formula is equivalent to a Boolean combination of atomic formulas. It is easy to check that the atomic formulas are all mutually algebraic and thus that the structure itself is. In contrast, it is possible to show that the structure $(\Q, \leq)$, despite having quantifier elimination, is not mutually algebraic (for example, one can show that the formula $x \leq y$ is not equivalent to a Boolean combination of mutually algebraic formulas).
\end{example}



\subsection{Quantifier elimination for mutually algebraic structures}
\label{sec:qe_ma}

We will make use of two quantifier elimination theorems for mutually algebraic structures. The first is due to Laskowski.

\begin{theorem}[\cite{laskowski2009elementary}, Theorem 4.2]
\label{thm:qe2}
If $M$ is mutually algebraic then every formula $\phi(\bar{x})$ is equivalent over $M$ to a Boolean combination of formulas of the form $\exists \bar{z}\, \theta(\bar{y}, \bar{z})$ (which may have parameters from $M$) where $\theta$ is quantifier free and mutually algebraic over $M$ and $\bar{y}$ is a subset of $\bar{x}$.
\end{theorem}

\begin{theorem}
\label{thm:qe1}
If $M$ is a mutually algebraic structure and $\phi(\bar{x})$ is mutually algebraic over $M$, then there is a quantifier free formula $\theta(\bar{x}, \bar{y})$ (which may have parameters from $M$) such that $\exists \bar{y}\, \theta(\bar{x}, \bar{y})$ is mutually algebraic over $M$ and $M \models \phi(\bar{x}) \to \exists\bar{y}\,\theta(\bar{x}, \bar{y})$.
\end{theorem}

The second theorem is a relatively straightforward consequence of the first one, together with some facts from the theory of mutual algebraicity. Our goal for the rest of this section is to give the proof. To do so, we will need a lemma about mutually algebraic formulas, due to Laskowski and Terry.

\begin{lemma}[\cite{laskowski2020uniformly}, Lemma A.1]
\label{lemma:ma1}
Suppose $M$ is a structure and $$\phi(\bar{x}) := \bigwedge_i \alpha_i(\bar{x}_i) \land \bigwedge_j\lnot\beta_j(\bar{x}_j)$$ is a formula such that
\begin{enumerate}
    \item $\phi(\bar{x})$ is mutually algebraic over $M$.
    \item $\{\bar{a} \mid M \models \phi(\bar{a})\}$ contains an infinite set of pairwise disjoint tuples.
    \item Each $\alpha_i(\bar{x}_i)$ and $\beta_j(\bar{x}_j)$ is mutually algebraic over $M$.
\end{enumerate}
Then $\alpha(\bar{x}) = \bigwedge_i \alpha_i(\bar{x}_i)$ is mutually algebraic over $M$.
\end{lemma}

Actually we need a slightly stronger version of this lemma. In particular, we need to replace the second condition on $\phi$ with the apparently weaker assumption that $\{\bar{a} \mid M \models \phi(\bar{a})\}$ is infinite. The next lemma, also due to Laskowski, tells us that since $\phi$ is mutually algebraic, the two conditions are actually equivalent.

\begin{lemma}[\cite{laskowski2009elementary}, Lemma 3.1]
\label{lemma:ma2}
Suppose $M$ is a structure and $\phi(\bar{x})$ is a formula which is mutually algebraic over $M$. If $\{\bar{x} \mid M \models \phi(\bar{a})\}$ is infinite then it contains an infinite set of pairwise disjoint tuples.
\end{lemma}

\noindent We can now prove Theorem~\ref{thm:qe1}.

\begin{proof}[Proof of Theorem~\ref{thm:qe1}]
By applying Laskowski's theorem and writing the resulting formula in disjunctive normal form, we get
\[
  M \models \phi(\bar{x}) \leftrightarrow \bigvee_{i} \left(\bigwedge_j \alpha_{i,j}(\bar{x}_{i, j}) \land \bigwedge_k \lnot \beta_{i, k}(\bar{x}_{i, k})\right)
\]
where each $\alpha_{i, j}(\bar{x}_{i, j})$ and each $\beta_{i, k}(\bar{x}_{i, k})$ is existential and mutually algebraic over $M$.

For each $i$, define 
\begin{align*}
\phi_i(\bar{x}) &:= \bigwedge_j \alpha_{i,j}(\bar{x}_{i, j}) \land \bigwedge_k \lnot \beta_{i, k}(\bar{x}_{i, k})\\
\alpha_i(\bar{x}) &:= \bigwedge_j \alpha_{i,j}(\bar{x}_{i, j})\\
A_i &= \{\bar{a} \mid M \models \phi_i(\bar{a})\}\\
\end{align*}
Note that since $\phi(\bar{x})$ is mutually algebraic and $M \models \phi_i(\bar{x}) \to \phi(\bar{x})$, $\phi_i(\bar{x})$ is also mutually algebraic. Thus by Lemma~\ref{lemma:ma1} above (or rather, its slightly strengthened version), we have that either $A_i$ is finite or $\alpha_i(\bar{x})$ is mutually algebraic.

In the former case, define $\gamma_i(\bar{x}) := \bigvee_{\bar{a} \in A_i}\bar{x} = \bar{a}$ and in the latter case, define $\gamma_i(\bar{x}) := \alpha_i(\bar{x})$. In either case, note that $\gamma_i$ is existential and mutually algebraic over $M$ and that $M \models \phi_i(\bar{x}) \to \gamma_i(\bar{x})$. Since $\phi(\bar{x})$ and $\bigvee_{i} \phi_i(\bar{x}_i)$ are equivalent in $M$, this gives us
\[
  M \models \phi(\bar{x}) \to \bigvee_i \gamma_i(\bar{x}).
\]
Since each $\gamma_i(\bar{x})$ is mutually algebraic, so is their disjunction. Pulling the existential quantifiers to the front, we have the desired formula.
\end{proof}

\section{The counterexample}
\label{sec:counterexample}

In this section we will describe the theory we use to answer Pakhomov's question. In order to do so, we need to fix a computable infinite binary tree $R$ with the property that none of its paths can be computably approximated. More precisely, say that a sequence $x \in \Cantor$ is \term{guessable} if there is an algorithm which, for each number $n$, enumerates a list of at most $O(n^2)$ strings of length $n$, one of which is $x\uh n$. We need a computable infinite binary tree $R$, none of whose paths are guessable.

It is not hard to directly construct such a tree $R$ but we can also simply pick a computable infinite binary tree whose paths are all Martin-L\"of random. Such a tree is known to exist\footnote{For example we can simply take the complement of any of the levels of the universal Martin-L\"of test.} and it is also easy to check that Martin-L\"of random sequences are not guessable. See the book \emph{Algorithmic Randomness and Complexity} by Downey and Hirschfeldt for more details about Martin-L\"of randomness~\cite{downey2010algorithmic}.

Essentially, our theory is the simplest theory all of whose models code an infinite path through $R$. We now give a more precise description.

\medskip\noindent\textbf{The language.} Let $\L$ be the language whose signature consists of:
\begin{enumerate}
    \item A constant symbol, $0$.
    \item Two unary function symbols, $S$ and $P$.
    \item A unary relation symbol, $A$.
\end{enumerate}
Also, although it is not officially part of the language $\L$, we will often use the following notation. Given any $n \in \N$,
\begin{itemize}
    \item $\lit{n}$ denotes the $\L$-term $S^{n}(0)$, e.g.\ $\lit{3}$ denotes $S(S(S(0)))$.
    \item $\lit{-n}$ denotes the $\L$-term $P^{n}(0)$, e.g.\ $\lit{-3}$ denotes $P(P(P(0)))$.
    \item $x + \lit{n}$ denotes the $\L$-term $S^n(x)$ and $x + \lit{-n}$ denotes the $\L$-term $P^{n}(x)$. We will also sometimes use $x - \lit{n}$ to denote $x + \lit{-n}$.
    \item We will often refer to $S$ as ``successor'' and $P$ as ``predecessor.''
\end{itemize}

\medskip\noindent\textbf{The theory.} Fix a computable infinite binary tree $R$, none of whose infinite paths are guessable, and let $T$ be the $\L$-theory consisting of:
\begin{enumerate}
    \item The theory of the integers with $0$, successor and predecessor, i.e.\ $\Th(\Z, 0, \Succ, \Pred)$.
    \item Axioms stating that $A$ (restricted to the elements $\lit{0}, \lit{1}, \lit{2}, \ldots$) describes a path through $R$. More precisely, for each $n \in \N$, $T$ contains the sentence
    \[
        \bigvee_{\sigma \in R_n} \bigg[\bigg(\bigwedge_{\sigma(i) = 0}\lnot A(\lit{i})\bigg)\land \bigg(\bigwedge_{\sigma(i) = 1} A(\lit{i}) \bigg)\bigg]
    \]
    where $R_n$ denotes the set of strings in $R$ of length $n$. 
\end{enumerate}
The second set of axioms ensures that from any model of $T$, we can computably recover a path through the tree $R$. We will now explain how this works. 

Given a sentence $\varphi$ and a model $M$, let's use the notation $\llbracket \varphi \rrbracket^M$ to denote \textbf{the truth-value of \boldmath$\varphi$ in \boldmath$M$}. We will often identify sequences of truth values with binary sequences by thinking of ``true'' as $1$ and ``false'' as $0$. Now suppose that $M$ is a model of $T$. We claim that the sequence $\tv{A(\lit{0})}{M}, \tv{A(\lit{1})}{M}, \tv{A(\lit{2})}{M}, \ldots$ is an infinite path through $R$. The point is that the axioms above guarantee that, for each $n \in \N$, the length $n$ initial segment of this sequence agrees with some \emph{specific} length $n$ string in $R$. Since all of its proper initial segments are in $R$, the sequence $\tv{A(\lit{0})}{M}, \tv{A(\lit{1})}{M}, \tv{A(\lit{2})}{M}, \ldots$ is indeed a path through $R$.

Note that this immediately implies that no model of $T$ is not computable---any such model computes an infinite path through $R$, but no such path is computable. In spite of this, we will see later that models of $T$ have quantifier elimination and so are very well-behaved in model-theoretic terms.

\medskip\noindent\textbf{Models of $T$.}
It will help to have a clear picture of the structure of models of $T$ and to fix some terminology for later. Since $T$ includes the theory of the integers with successor and predecessor, $T$ proves that $S$ and $P$ are injective functions with no cycles and that they are inverses. Thus any model of $T$ consists of a disjoint union of one or more $\Z$-chains, with $S$ moving forward along each chain, $P$ moving backward and the constant $0$ sitting in the middle of one of the chains. There is also a well-defined notion of distance: the distance between two elements of the same chain is simply the number of steps apart they are on the chain (and the distance between elements of two different chains is $\infty$).

Furthermore, each element of each chain is labelled with a truth value (corresponding to whether the predicate $A$ holds of that element or not) and thus each chain gives rise to a bi-infinite binary sequence. If we start at the element $0$ and move forward along its chain, then, as we saw above, the binary sequence we get is guaranteed to be a path through the tree $R$.

Given a model $M$ of $T$ and elements $a, b \in M$, we will use the following terminology.
\begin{itemize}
    \item The \term{signed distance} from $a$ to $b$ is the unique integer $k$ (if it exists) such that $b = a + \lit{k}$. If no such $k$ exists then the signed distance is $\infty$.
    \item The \term{distance between} $a$ and $b$ is the absolute value of the signed distance (where the absolute value of $\infty$ is $\infty$).
    \item For $k \in \N$, the \term{\boldmath$k$-neighborhood} of $a$ is the set $\{a - \lit{k}, a - \lit{(k - 1)}, \ldots, a + \lit{k}\}$.
\end{itemize}
Note that if the signed distance from $a$ to $b$ is $k < \infty$, the signed distance from $b$ to $a$ is $-k$.

\begin{remark}
\label{remark:generic}
By choosing a somewhat more complicated theory, it is possible to simplify some of the proofs later in this paper. In particular, we can add axioms to $T$ which state that $A$ behaves \emph{generically}, in the sense that every finite pattern of values of $A$ occurs somewhere. More precisely, for every finite binary string $\sigma \in 2^{<\omega}$ we add the axiom
\[
    \exists x\bigg[\bigg(\bigwedge_{\sigma(i) = 0} \lnot A(x + \lit{i})\bigg) \land \bigg(\bigwedge_{\sigma(i) = 1}A(x + \lit{i})\bigg)\bigg].
\]
Equivalently, we can replace $T$ with its model completion. Making this change would allow us to simplify the proofs of Propositions~\ref{prop:indiscernability} and~\ref{prop:satisfaction1} and Lemma~\ref{lemma:ma_close}.
\end{remark}

\section{Proof of the main theorem}

Let $\L$ and $T$ be the language and theory described in the previous section. In this section, we will prove that no theory which is definitionally equivalent to $T$ has a computable model. Since $T$ is a consistent, c.e.\ theory, this is enough to prove Theorem~\ref{thm:main}. 

In order to prove this, let's fix a language $\L'$ and an $\L'$-theory $T'$ which is definitionally equivalent to $T$. Note that since the language $\L$ has finite signature, we may assume that $\L'$ does as well.\footnote{The point is that if a theory $T$ is in a language with finite signature and $T'$ is any theory definitionally equivalent to $T$ then $T'$ has a subtheory in a language with finite signature which is also definitionally equivalent to $T$.} Now fix a model $M$ of $T'$. Our goal is to prove that $M$ is not computable.\footnote{Recall that a model is computable if its underlying set is $\N$ and all of its functions and relations are computable as functions or relations on $\N$. Note that since we are assuming $\L'$ has finite signature, we don't need to worry about whether these functions and relations are uniformly computable.} 



Before beginning, it will be useful to fix a few conventions. First, recall from section~\ref{sec:de} that $M$ gives rise to a model $\unprime{M}$ of $T$ which has the same underlying set and the same algebra of definable sets as $M$. We will often abuse notation slightly and use $M$ to refer to both $M$ itself and $\unprime{M}$. For example, if $\phi$ is an $\L$-formula, we will use $M \models \phi(\bar{a})$ to mean $\unprime{M} \models \phi(\bar{a})$. Also, we will say things like ``$b$ is the successor of $a$'' to mean $\unprime{M} \models b = S(a)$. Second, unless explicitly stated otherwise, we assume that formulas do not contain parameters.

\subsection{Proof strategy}

To prove that $M$ is not computable, we will show that the sequence $\llbracket A(\lit{0})\rrbracket^M, \llbracket A(\lit{1})\rrbracket^M, \ldots$ is guessable (in the sense of section~\ref{sec:counterexample}) relative to an oracle for $M$. Since the axioms of $T$ ensure that this sequence is a path through the tree $R$, and hence not guessable, this is enough to show that $M$ is not computable.

To show that the sequence $\llbracket A(\lit{0})\rrbracket^M, \llbracket A(\lit{1})\rrbracket^M, \ldots$ is guessable from an oracle for $M$, we will first prove that $M$ is mutually algebraic. To do so, we will essentially show that models of $T$ have quantifier elimination and use this to prove that $\unprime{M}$ is mutually algebraic. The mutual algebraicity of $M$ itself follows because mutual algebraicity is preserved under definitional equivalence (because mutual algebraicity depends only on the algebra of definable sets, which is itself preserved under definitional equivalence).

Once we know that $M$ is mutually algebraic, we can apply the quantifier elimination results of section~\ref{sec:qe_ma} to infer that $S$ and $A$ are close to being quantifier-free definable in $M$. In particular, the formula $S(x) = y$ is mutually algebraic and so, by Theorem~\ref{thm:qe1}, there is an existential $\L'$-formula $\psi_S(x, y)$ such that $\psi_S$ is mutually algebraic and $M \models S(x) = y \to \psi_S(x, y)$. 

We can think of $\psi_S$ as a multi-valued function which takes each element $a \in M$ to the set of elements $b \in M$ such that $M \models \psi_S(a, b)$. Since $\psi_S$ is an existential formula, the graph of this multi-valued function is computably enumerable from an oracle for $M$. Since $\psi_S$ is mutually algebraic, there are only finitely many elements in the image of each $a$. And since $M \models S(x) = y \to \psi_S(x, y)$, the successor of $a$ is always in the image of $a$. Putting this all together, we can think of this multi-valued function as giving us, for each $a \in M$, a finite list of guesses for $S(a)$ which is computably enumerable relative to an oracle for $M$.

To finish the proof, we can leverage our ability to enumerate a finite list of guesses for the successor of each element to enumerate a short list of guesses for each initial segment of the sequence $\llbracket A(\lit{0}) \rrbracket^M, \llbracket A(\lit{1}) \rrbracket^M, \ldots$. To accomplish this, we will have to make use of our understanding of the structure of definable subsets of $M$, which we first develop in order to prove mutual algebraicity.

\subsection{Model-theoretic tameness of \texorpdfstring{$M$}{M}}

Our first goal is to prove that $M$ is mutually algebraic. One way to do this is to show that models of $T$ satisfy quantifier elimination and then note that all atomic $\L$-formulas are mutually algebraic over $\unprime{M}$---this implies that $\unprime{M}$ is mutually algebraic and hence that $M$ is as well. However, it will be helpful for us later to have a more detailed understanding of the structure of definable subsets of $M$. Thus, instead of just proving quantifier elimination for models of $T$, we will prove a stronger statement, which is essentially a quantitative version of quantifier elimination.



To explain this stronger statement, let's first consider the meaning of quantifier elimination in models of $T$. By examining the atomic formulas of $\L$, we can see that it means that for every $\L$-formula $\phi(\bar{x})$ and tuple $\bar{a}$, the truth of $\phi(\bar{a})$ depends only on which elements of $\bar{a}$ are close to each other (and to $0$), how close they are, and the values of the predicate $A$ in a small neighborhood of each element. In our stronger statement, we will quantify exactly what ``close'' and ``small'' mean in this description. We will also extend this to $\L'$-formulas. We will refer to the resulting statement as the \term{indiscernability principle} for $M$. In order to make all of this precise, we first need to introduce some terminology.


\medskip\noindent\textbf{The radius of a formula.}
For any $\L$-formula $\phi$ written in prenex normal form, inductively define the \term{radius} of $\phi$, written $\rad(\phi)$, as follows.
\begin{enumerate}
    \item If $\phi$ is quantifier free then $\rad(\phi)$ is the total number of occurrences of $S$ and $P$ in $\phi$.
    \item If $\phi$ has the form $\exists x\,\psi$ or $\forall x\, \psi$ then $\rad(\phi) = 2\cdot\rad(\psi)$.
\end{enumerate}
If $\phi$ is an $\L'$-formula in prenex normal form then we define $\rad(\phi)$ in a similar way except that we change the case where $\phi$ is quantifier free to define $\rad(\phi)$ to be $\rad(\unprime{\phi})$ (after first putting $\unprime{\phi}$ in prenex normal form). The idea of the radius of a formula is that in the description of quantifier elimination for $M$ above, we should interpret ``close'' to mean ``within distance $\rad(\phi)$.''

\medskip\noindent\textbf{The \boldmath$r$-type of a tuple.}
Given a tuple $\bar{a} = (a_1,\ldots,a_n)$ in $M$ and a number $r \in \N$, define:
\begin{itemize}
    \item The \term{\boldmath$r$-distance table} of $\bar{a}$ records the signed distances between the coordinates of $\bar{a}$ and between each coordinate of $\bar{a}$ and $0$, treating any distance greater than $r$ as $\infty$. More precisely, it is the function $f \colon \{0,1,\ldots, n\}^2 \to \{-r, -(r - 1), \ldots, r, \infty\}$ such that if the distance between $a_i$ and $a_j$ is at most $r$ then $f(i, j)$ is the signed distance from $a_i$ to $a_j$ and otherwise $f(i, j) = \infty$ (and where we interpret $a_0$ as $0$).
    \item The \term{\boldmath$r$-neighborhood type} of any element $a \in M$ is the sequence of truth values $\llbracket A(a - \lit{r})\rrbracket^M, \llbracket A(a - \lit{r - 1})\rrbracket^M, \ldots, \llbracket A(a + \lit{r})\rrbracket^M$.
    \item The \term{\boldmath$r$-type} of $\bar{a}$ is the $r$-distance table of $\bar{a}$ together with the sequence recording the $r$-neighborhood type of each coordinate of $\bar{a}$.
\end{itemize}

\medskip\noindent\textbf{The indiscernability principle.} We can now state a formal version of the indiscernability principle described above.

\begin{proposition}
\label{prop:indiscernability}
If $\phi$ is an $\L$-formula in prenex normal form and of radius $r$ and $\bar{a}, \bar{b}$ are tuples in $M$ with the same $r$-type then $M \models \phi(\bar{a})$ if and only if $M \models \phi(\bar{b})$.
\end{proposition}

\begin{proof}
By induction on the number of quantifiers in $\phi$. For quantifier free formulas, this is easy to verify. If $\phi$ has quantifiers then it suffices to assume $\phi = \exists x\, \psi$ since the case of a universal quantifier is symmetric (i.e.\ by considering $\lnot\phi$ instead of $\phi$ and pushing the negation past the quantifiers to get it into prenex normal form). Also, it's enough to assume $M \models \phi(\bar{a})$ and prove $M \models \phi(\bar{b})$---the other direction also follows by symmetry.

So let's assume that $M \models \exists x\, \psi(\bar{a}, x)$. Thus there is some $c$ such that $M \models \psi(\bar{a}, c)$. We need to find some $d$ such that $M \models \psi(\bar{b}, d)$. Note that it is enough to find $d$ such that $\bar{a}c$ and $\bar{b}d$ have the same $r/2$-type, because if this holds then we can apply the induction hypothesis to $\psi$ to get that $M \models \psi(\bar{b}, d)$.

There are two cases depending on whether $c$ is close to any element of $\bar{a}$ or not. Also to reduce casework, we adopt the convention that $a_0 = b_0 = 0$ (note that this does not change the fact that $\bar{a}$ and $\bar{b}$ have the same $r$-type).

\medskip\noindent\textbf{Case 1.} First suppose that $c$ is distance at most $r/2$ from some coordinate of $\bar{a}$. In particular, there is some $i \leq n$ and $-r/2 \leq k \leq r/2$ such that $c = a_i + \lit{k}$. In this case, we can pick $d$ to be close to the corresponding element of $\bar{b}$, i.e.\ $d = b_i + \lit{k}$. We claim that $\bar{a}c$ and $\bar{b}d$ have the same $r/2$-type.

First, we need to check that the $r/2$-distance tables are the same. It suffices to check that for each $j$, either $a_j, c$ and $b_j, d$ have the same signed distance or both have distance greater than $r/2$. Suppose that $a_j = c + \lit{k'}$ for some integer $-r/2 \leq k' \leq r/2$. By substitution, $a_j = (a_i + \lit{k}) + \lit{k'} = a_i + \lit{k + k'}$. Since $|k + k'| \leq r$ and since $\bar{a}, \bar{b}$ have the same $r$-distance table, this implies that $b_j = b_i + \lit{k + k'}$ and hence that $b_j = d + \lit{k'}$. The other cases can be handled similarly.

Second, we need to check that the $r/2$-neighborhood type of $c$ is the same as that of $d$. This follows from the fact that the $r/2$-neighborhood of $c$ is contained in the $r$-neighborhood of $a_i$, the $r/2$-neighborhood of $d$ is contained in the $r$-neighborhood of $b_i$ and the $r$-neighborhood types of $a_i$ and $b_i$ are the same.

\medskip\noindent\textbf{Case 2.} Now suppose that $c$ is distance more than $r/2$ from every coordinate of $\bar{a}$. It is enough to find some $d$ which has the same $r/2$-neighborhood type as $c$ and which is distance more than $r/2$ from every coordinate of $\bar{b}$. The point is that for such a $d$, it is easy to see that $\bar{a}c$ and $\bar{b}d$ have the same $r/2$-type.

We now claim that some such $d$ must exist.\footnote{This case becomes more or less trivial if $T$ is modified in the way described in Remark~\ref{remark:generic}. This is because the existence of such an element $d$ is guaranteed by the extra axioms described in that remark.} Suppose for contradiction that this is false. Then every element of $M$ with the same $r/2$-neighborhood type as $c$ must be contained in the $r/2$ neighborhood of some element of $\bar{b}$. In particular, this implies that there are a finite number of such elements and they all have the form $b_i + \lit{k}$ for some $i \leq n$ and $-r/2 \leq k \leq r/2$.

Suppose there are exactly $m$ such elements and they are equal to $b_{i_1} + \lit{k_1},\ldots,b_{i_m} + \lit{k_m}$ (where for each $j$, $-r/2 \leq k_j \leq r/2$). It follows from the fact that $\bar{a}$ and $\bar{b}$ have the same $r$-type that the corresponding elements $a_{i_1} + \lit{k_1}, \ldots, a_{i_m} + \lit{k_m}$ are also all distinct and have the same $r/2$-neighborhood type as $c$. However, since only $m$ elements of $M$ have this $r/2$-neighborhood type, $c$ must be among this list of elements, which contradicts the assumption that $c$ is not within distance $r/2$ of any coordinate of $\bar{a}$.
\end{proof}

\begin{corollary}
Proposition~\ref{prop:indiscernability} also holds for all $\L'$-formulas in prenex normal form.
\end{corollary}

\begin{proof}
Suppose $\phi$ is an $\L'$-formula of radius $r$ and that $\bar{a}, \bar{b}$ are tuples in $M$ with the same $r$-type. In the case where $\phi$ is quantifier free, the radius of $\unprime{\phi}$ is also $r$, for the trivial reason that radius of a quantifier-free $\L'$-formula is defined as the radius of its $\L$-translation. Hence, we can apply the indiscernability principle to $\unprime{\phi}$ to get
\begin{align*}
M \models \phi(\bar{a}) &\iff M \models \unprime{\phi}(\bar{a})\\
&\iff M \models \unprime{\phi}(\bar{b}) \iff M \models \phi(\bar{b}).
\end{align*}
When $\phi$ has quantifiers, the inductive argument that we gave in the proof of Proposition~\ref{prop:indiscernability} still works.
\end{proof}

\medskip\noindent\textbf{\boldmath$M$ is mutually algebraic.}
For a fixed $r$-type, the assertion that a tuple $\bar{x} = (x_1,\ldots,x_n)$ has that $r$-type is expressible as a Boolean combination of $\L$-formulas of the following forms.
\begin{enumerate}
    \item $x_i = x_j + \lit{k}$ for some indices $i, j \leq n$ and some $-r \leq k \leq r$.  \item $x_i = \lit{k}$ for some index $i \leq n$ and some $-r \leq k \leq r$.
    \item $A(x_i + \lit{k})$ for some index $i \leq n$ and some $-r \leq k \leq r$.
\end{enumerate}
It is easy to check that each type of formula listed above is mutually algebraic over $M$ (for the second and third there is actually nothing to check because they both involve only one free variable). Furthermore, for any fixed $r$, there are a finite number of possible $r$-types. Thus the indiscernability principle implies that every $\L$-formula $\phi$ is equivalent to a finite conjunction of Boolean combinations of mutually algebraic $\L$-formulas (namely a conjunction over all $r$-types that satisfy $\phi$). 

This shows that $M$ is mutually algebraic when considered as an $\L$-structure (i.e.\ that $\unprime{M}$ is mutually algebraic). However, it is easy to conclude that $M$ is also mutually algebraic when considered as an $\L'$-structure. For a given formula $\phi$, we know from our reasoning above that $\unprime{\phi}$ is equivalent to a Boolean combination of mutually algebraic $\L$-formulas. Next, we can replace each formula in this Boolean combination by its corresponding $\L'$-formula. Since the mutual algebraicity of a formula only depends on the set that it defines, and since this is invariant under translating between $\L$ and $\L'$, we conclude that $\phi$ is equivalent to a Boolean combination of mutually algebraic $\L'$-formulas.

\begin{remark}
The reasoning above also shows that $M$ has quantifier elimination when considered as an $\L$-structure (i.e.\ $\unprime{M}$ has quantifier elimination). The point is just that a tuple having a certain $r$-type is expressible as a quantifier free $\L$-formula.
\end{remark}

\subsection{The satisfaction algorithm}

We will now explain how the indiscernability principle implies that the satisfaction relation for $\L'$-formulas over $M$ is very nearly computable relative to an oracle for $M$. At the end of this subsection, we will explain why this is useful.

The main idea (of computing the satisfaction relation) is that to check whether $M \models \exists x\, \phi(\bar{a}, x)$, we don't need to try plugging in every element of $M$ for $x$, just those elements which are close to some coordinate of $\bar{a}$ (or to $0$), plus one element of each possible $\rad(\phi)$-neighborhood type which is far from all the coordinates of $\bar{a}$. In other words, checking the truth of an existential formula can be reduced to checking the truth of a finite number of atomic formulas. This intuition is formalized by the next proposition, whose proof essentially just consists of this idea, but with a number of messy details in order to make precise the idea of trying all the different $\rad(\phi)$-neighborhood types which are far from elements of $\bar{a}$.

\begin{proposition}[Satisfaction algorithm for existential formulas]
\label{prop:satisfaction1}
Suppose $\phi(\bar{x})$ is an existential $\L'$-formula with radius $r$. There is an algorithm which, given a tuple $\bar{a}$ in $M$ and the following data
\begin{enumerate}
    \item an oracle for $M$
    \item and a finite set $U \subseteq M$,
\end{enumerate}
tries to check whether $M \models \phi(\bar{a})$. Furthermore, if $U$ contains the $r$-neighborhood of every coordinate of $\bar{a}$ then the output of the algorithm is correct.
\end{proposition}

\begin{proof}
Let $\theta(\bar{x}, \bar{y})$ be a quantifier free formula such that $\phi(\bar{x}) = \exists\bar{y}\, \theta(\bar{x}, \bar{y})$ and let $n = |\bar{x}|$ and $m = |\bar{y}|$ (i.e.\ the number of free and bound variables in $\phi$, respectively). Next, fix a finite set $V$ such that for each possible $r$-neighborhood type $p$, $V$ contains at least $(2r + 1)(n + m + 1)$ points of type $p$ (or if fewer than $(2r + 1)(n + m + 1)$ points have $r$-neighborhood type $p$ then $V$ contains every such point).\footnote{The idea is that this number is big enough that if we have any other $n + m + 1$ points then at least one point in $V$ which is of $r$-neighborhood type $p$ will be distance more than $r$ from all these $n + m + 1$ points.} Also $V$ should contain $0$. Let $V'$ be the set consisting of all elements within distance $r$ of some element of $V$. Note that since $V'$ is finite, we can ``hard-code'' it into our algorithm.

\medskip\noindent\textit{Algorithm description.} 
To check if $M \models \phi(\bar{a})$, look at each tuple $\bar{b}$ of elements of $U \cup V'$ and check if $M \models \theta(\bar{a}, \bar{b})$. If this occurs for at least one such $\bar{b}$ then output ``true.'' Otherwise, output ``false.'' Note that checking the truth of a quantifier free formula (such as $\theta$) is computable from an oracle for $M$.

\medskip\noindent\textit{Verification.}
Let's assume that $U$ contains the $r$-neighborhood of each coordinate of $\bar{a}$ and check that the output of the algorithm is correct. It is obvious that the algorithm has no false positives: if $M \models \theta(\bar{a}, \bar{b})$ for some $\bar{b}$ then $M \models \phi(\bar{a})$. Thus it suffices to assume that $M \models \phi(\bar{a})$ and show that there is some tuple $\bar{b}$ in $U\cup V'$ such that $M \models \theta(\bar{a}, \bar{b})$.

To accomplish this, we will pick elements of $\bar{b}$ one at a time and, at each step, ensure that all the elements we have picked so far come from the set $U \cup V'$. More precisely, we will pick elements $b_1,\ldots,b_m$ such that for each $i \leq m$,
\[
    M \models \exists y_{i + 1}\ldots \exists y_m \,\theta(\bar{a}, b_1,\ldots,b_i, y_{i + 1},\ldots,y_m)
\]
and we will try to ensure that for each $i$, $b_i \in U \cup V'$. However, in order to do this, we will need a somewhat stronger inductive assumption.

Let's first explain on an informal level how the induction works and why we need a stronger inductive assumption. On the first step of the induction, things work pretty well. It is possible to use the indiscernability principle to show that we can pick some $b_1$ which satisfies the condition above and which is close to some element of either $\bar{a}$ or $V$. Since $U$ contains a reasonably large neighborhood around each element of $\bar{a}$ and $V'$ contains a reasonably large neighborhood around each element of $V$, this means we can pick $b_1$ from $U\cup V'$. On the second step of the induction, however, things start to go wrong. We can again use the indiscernability principle to show that we can pick some $b_2$ which satisfies the condition above and which is close to either $b_1$ or to some element of either $\bar{a}$ or $V$. In the latter case, there is no problem: we can still pick $b_2$ from $U\cup V'$. But in the former case, there may be a problem. If the element $b_1$ we picked on the first step happens to be near the ``boundary'' of $U\cup V'$ then even a $b_2$ which is relatively close to it might no longer be inside $U\cup V'$.

We can fix this problem by requiring not just that $b_1$ is in $U\cup V'$, but also that it is far from the ``boundary'' of $U\cup V'$. In other words, we need to require that $b_1$ is close to $\bar{a}$ or $V$ in some stronger way than simply requiring that it be in $U\cup V'$. In fact, it is enough to require that $b_1$ be within distance $r/2$ of some element of $\bar{a}$ or $V$ and more generally, that each $b_i$ is within distance $r/2 + \ldots + r/2^i$ of some element of $\bar{a}$ or $V$.

To state this formally, we define sets $W_0 \subseteq W_1 \subseteq W_2 \subseteq \ldots \subseteq W_m$ as follows. $W_0$ consists of the coordinates of $\bar{a}$ together with the elements of $V$. For each $0 < i \leq m$, $W_i$ consists of all points in $M$ which are within distance $r/2^i$ of some element of $W_{i - 1}$ (note that this is equivalent to being within distance $r/2 + r/4 + \ldots + r/2^i$ of some element of $W_0$). Note that by assumption, $U \cup V'$ contains the $r$-neighborhood of each element of $W_0$. It follows that each $W_i$ is contained in $U \cup V'$

Also, define a sequence of formulas $\phi_0, \phi_1,\ldots,\phi_m$ by removing the quantifiers from $\phi$ one at a time. More precisely, define
\[
    \phi_i(\bar{x}, y_1,\ldots,y_i) := \exists y_{i + 1} \, \ldots, \exists y_m \theta(\bar{x}, \bar{y}).
\]
So, for example,
\begin{itemize}
    \item $\varphi_0(\bar{x}) = \exists y_1 \dots \exists y_m \theta(\bar{x},\bar{y})=\varphi(\bar{x})$
    \item $\varphi_1(\bar{x},y_1) = \exists y_2 \dots \exists y_m \theta(\bar{x},\bar{y})$
    \item $\varphi_2 (\bar{x},y_1,y_2) = \exists y_3 \dots \exists y_m \theta(\bar{x},\bar{y})$
    \item \dots 
    \item $\varphi_m (\bar{x},y_1,\dots,y_m) = \theta(\bar{x},\bar{y})$.
\end{itemize}

We will now inductively construct a sequence of points $b_1,\ldots,b_m$ such that for each $i$, $b_i \in W_i$ and $M \models \phi_i(\bar{a}, b_1,\ldots,b_i)$. Since $W_m \subseteq U \cup V'$ and $\phi_m = \theta$, this is sufficient to finish the proof.

The base case of this induction is simply the assertion that $M \models \phi(\bar{a})$ which we assumed above. Now assume that we have already found $b_1,\ldots,b_i$ and we will show how to find $b_{i + 1}$. Since $M \models \phi_i(\bar{a}, b_1,\ldots,b_i)$, there is some $c$ such that $M \models \phi_{i + 1}(\bar{a}, b_1,\ldots,b_i, c)$. The idea is that we can pick $b_{i + 1}$ by mimicking $c$. If $c$ is within distance $r/2^{i + 1}$ of some coordinate of $\bar{a}$, $0$ or some $b_j$ for $j \leq i$ then we set $b_{i + 1} = c$. Otherwise, we can pick $b_{i + 1}$ to be some element of $V$ with the same $r$-neighborhood type as $c$ and which is also distance at least $r/2^{i + 1}$ from all coordinates of $\bar{a}$, $0$ and all $b_j$. We can do this because either $V$ contains many points of that $r$-neighborhood type (more than all the points within distance $r/2^{i + 1}$ of $\bar{a}$, $0$ and $b_1,\ldots, b_i$---this is why we chose the number $(2r + 1)(n + m + 1)$) or there are not very many such points and $V$ contains $c$ itself. Note that in the first case, $b_{i + 1}$ is within distance $r/2^{i + 1}$ of some element of $W_i$, and in the second case, $b_{i + 1} \in V$. Thus in either case $b_{i + 1} \in W_{i + 1}$.

Also, note that in either case $\bar{a}b_1\ldots b_i c$ and $\bar{a}b_1\ldots b_i b_{i + 1}$ have the same $r/2^{i + 1}$-type. Since the radius of $\phi_{i + 1}$ can be seen to be $r/2^{i + 1}$ and $M \models \phi_{i + 1}(\bar{a}, b_1,\ldots,b_i, c)$, the indiscernability principle implies that $M \models \phi_{i + 1}(\bar{a},b_1,\ldots,b_{i + 1})$, as desired.
\end{proof}

We now want to give an algorithm to compute the satisfaction relation of an arbitrary formula. One way to do this is to recursively apply the idea of Proposition~\ref{prop:satisfaction1} to reduce checking the truth of a formula with an arbitrary number of quantifiers to checking the truth of a finite number of atomic formulas. However, if we invoke the quantifier elimination results of section~\ref{sec:qe_ma} then we can do something simpler. Recall that Theorem~\ref{thm:qe2} tells us every formula is equivalent over $M$ to a Boolean combination of existential formulas. Thus the algorithm for existential formulas almost immediately yields an algorithm for arbitrary formulas.

\begin{proposition}[Satisfaction algorithm for arbitrary formulas]
\label{prop:satisfaction}
Suppose $\phi(\bar{x})$ is an $\L'$-formula. There is a number $r \in \N$ and an algorithm which, given any tuple $\bar{a}$ in $M$ and the following data
\begin{enumerate}
    \item an oracle for $M$
    \item and a finite set $U \subseteq M$,
\end{enumerate}
tries to check whether $M \models \phi(\bar{a})$. Furthermore, if $U$ contains the $r$-neighborhood of every coordinate of $\bar{a}$ then the algorithm is correct.
\end{proposition}

\begin{definition}
For convenience, we will refer to the number $r$ in the statement of this proposition as the \term{satisfaction radius} of $\phi$.
\end{definition}

\begin{proof}
By Theorem~\ref{thm:qe2}, $\phi(\bar{x})$ is equivalent over $M$ to a Boolean combination of existential $\L'$-formulas, $\psi_1(\bar{x}),\ldots, \psi_m(\bar{x})$ (which may have parameters from $M$). Let $r_1,\ldots,r_m$ denote the radii of these formulas and let $r = \max(r_1,\ldots,r_m)$.

The algorithm is simple to describe, but is made slightly more complicated by the fact that the formulas $\psi_i$ may contain parameters from $M$. For clarity, we will first assume that they do not contain such parameters and then explain how to modify the algorithm in the case where they do.

Here's the algorithm (in the case where there are no parameters). For each $i \leq m$, use the algorithm for existential formulas and the set $U$ to check the truth of $\psi_i(\bar{a})$. Then assume all the reported truth values are correct and use them to compute the truth value of $\phi(\bar{a})$.

If $U$ contains an $r$-neighborhood around every coordinate of $\bar{a}$ then for each $i \leq m$, it contains an $r_i$-neighborhood around each coordinate of $\bar{a}$. So in this case, the truth values we compute for $\psi_1(\bar{a}), \ldots, \psi_m(\bar{a})$ are guaranteed to be correct and thus the final truth value for $\phi(\bar{a})$ is also correct.

Now suppose that the formulas $\psi_i$ contain parameters from $M$. Let $\bar{b}_i$ be the tuple of parameters of $\psi_i$. Let $V$ be the set containing the $r$-neighborhood of each element of each tuple of parameters $\bar{b}_i$. The only modification that is needed to the algorithm described above is that instead of using $U$ itself, we should use $U \cup V$ when applying the satisfaction algorithm for existential formulas (and note that since $V$ is finite, we can simply hard-code it into our algorithm).
\end{proof}


Here's why this algorithm is useful. Note that if we had some way of computably generating the set $U$ then we would be able to outright compute the satisfaction relation for $\phi$ using just an oracle for $M$. In turn, this would allow us to use an oracle for $M$ to compute the sequence $\tv{A(\lit{0})}{M}, \tv{A(\lit{1})}{M}, \tv{A(\lit{2})}{M},\ldots$, which is a path through $R$. Since $R$ has no computable paths, this would imply $M$ is not computable. Thus to finish our proof of the uncomputability of $M$, it is enough to find an algorithm for generating the set $U$ needed by the satisfaction algorithm. Actually, we can't quite do this in general, but we can do something almost as good: we can enumerate a short list of candidates for $U$. This is enough to show that the sequence $\tv{A(\lit{0})}{M}, \tv{A(\lit{1})}{M}, \tv{A(\lit{2})}{M},\ldots$ is guessable from an oracle for $M$. Since $R$ has no guessable paths, this is still enough to imply that $M$ is not computable.

\subsection{The guessing algorithm}

We will now prove that $M$ is not computable. As discussed above, we will do so by proving that the sequence $\tv{A(\lit{0})}{M}, \tv{A(\lit{1})}{M}, \tv{A(\lit{2})}{M}, \ldots$ is guessable relative to an oracle for $M$. Since the axioms of $T$ ensure that this sequence is a path through $R$ and since no path through $R$ is guessable, this implies that $M$ is not computable.

In other words, we can complete our proof by constructing an algorithm which, given an oracle for $M$ and a number $n$, enumerates a list of at most $O(n^2)$ guesses (at least one of which is correct) for the finite sequence $\llbracket A(\lit{0})\rrbracket^M, \llbracket A(\lit{1})\rrbracket^M, \ldots, \llbracket A(\lit{n}) \rrbracket^M$. 

The rest of this section is devoted to constructing this algorithm. Since it would become annoying to append the phrase ``relative to an oracle for $M$'' to every other sentence that follows, we will adopt the convention that we always implicitly have access to an oracle for $M$, even if we do not say so explicitly. Thus whenever we say that something is computable or computably enumerable, we mean relative to an oracle for $M$.



\medskip\noindent\textbf{Warm-up: when \boldmath$S$ has a quantifier free definition.}
We will begin by constructing an algorithm for one especially simple case. Note that this case is included only to demonstrate how the satisfaction algorithm can be used and to motivate the rest of the proof; it can be skipped without missing any essential details.


The ``especially simple case'' we are referring to is the case in which $S$ has a quantifier free $\L'$-definition. We will see that in this case, the sequence $\llbracket A(\lit{0})\rrbracket^M, \llbracket A(\lit{1})\rrbracket^M, \ldots $ is not only guessable, but actually computable.


To begin, let $\phi_S(x, y)$ be the $\L'$-definition of $S$---i.e.\ for every $a, b \in M$, $M \models S(a) = b$ if and only if $M \models \phi_S(a, b)$. Note that since $\phi_S$ is quantifier-free, the successor function in $M$ is computable: to find $S(a)$ we can just enumerate elements of $M$ until we see an element $b$ such that $M \models \phi_S(a, b)$ (which we can check because $\phi_S$ is quantifier-free). Likewise, we can also compute the predecessor function: instead of waiting for an element $b$ such that $M \models \phi_S(a, b)$, we wait for an element $b$ such that $M \models \phi_S(b, a)$.

We can now explain how to compute $\llbracket A(\lit{0})\rrbracket^M, \llbracket A(\lit{1})\rrbracket^M, \ldots $. Let $\phi_A(x)$ be the $\L'$-definition of $A$ and let $r$ be the satisfaction radius of $\phi_A$. Given a number $n$, do the following.
\begin{enumerate}
    \item First use the fact that the successor function is computable to compute $\lit{n} = S^n(0)$.
    \item Next, use the fact that the successor and predecessor functions are computable to compute the $r$-neighborhood of $\lit{n}$. Let $U$ denote the set of elements in this $r$-neighborhood.
    \item Finally, use the satisfaction algorithm for $\phi_A$, along with the set $U$, to check whether $M \models \phi_A(\lit{n})$ and output the result as the truth value of $A(\lit{n})$. Note that since $U$ really does contain the $r$-neighborhood of $\lit{n}$, the outcome of this step is guaranteed to be correct.
\end{enumerate}

\medskip\noindent\textbf{Idea of the full algorithm.}
We have just seen an algorithm that computes the sequence $\llbracket A(\lit{0})\rrbracket^M, \llbracket A(\lit{1})\rrbracket^M, \ldots $ (without needing to make guesses) in the special case where $S$ is definable by a quantifier-free $\L'$-formula. We can no longer assume that there is a quantifier-free definition of $S$, but by applying the quantifier elimination theorem for mutually algebraic formulas over mutually algebraic structures from section~\ref{sec:qe_ma}, we have something almost as good. Namely, let $\phi_S(x, y)$ be the $\L'$-definition of $S$. It is easy to see that $\phi_S$ is mutually algebraic and so, by Theorem~\ref{thm:qe1}, there is a mutually algebraic existential formula $\psi_S(x, y)$ (possibly with parameters from $M$) such that $M \models \phi_S(x, y) \to \psi_S(x, y)$. 

The formula $\psi_S(x, y)$ should be thought of as an ``approximation'' to the successor relation in $M$. In particular, for a fixed element $a$, any $b$ such that $M \models \psi_S(a, b)$ holds should be thought of as a candidate for $S(a)$ and any $b$ such that $M \models \psi_S(b, a)$ holds should be thought of as a candidate for $P(a)$. This is justified by the following two facts.
\begin{enumerate}
    \item Since $M \models \phi_S(x, y) \to \psi_S(x, y)$, we have $M \models \psi_S(a, S(a))$ and $M \models \psi_S(P(a), a)$. In other words, the candidates for the successor and predecessor of $a$ include the true successor and predecessor of $a$, respectively.
    \item Since $\psi_S$ is mutually algebraic, there are not very many such candidates.
\end{enumerate}
The core idea of the algorithm is that since $\psi_S(x, y)$ is existential, the set of candidates for $S(a)$ and $P(a)$ is computably enumerable: to check if $M \models \psi_S(a, b)$, we simply wait until we see some tuple in $M$ which can serve as a witness. Thus we have an algorithm which, given any $a \in M$, enumerates a short list of candidates for $S(a)$ and $P(a)$.

Next, we can bootstrap this into an algorithm which, for any $a \in M$ and any number $n \in \N$, enumerates a list of guesses for the sequence $a - \lit{n}, a - \lit{(n - 1)}, \ldots, a + \lit{n}$: basically, enumerate guesses for the successor and predecessor of $a$, then enumerate guesses for the successor and predecessor of each of those guesses and so on, for $n$ rounds. This puts us in a situation much like the previous subsection (where the successor and predecessor functions were computable). In particular, we can enumerate guesses for the sequence $\llbracket A(\lit{0})\rrbracket^M, \ldots, \llbracket A(\lit{n})\rrbracket^M$ as follows.
\begin{enumerate}
    \item First, let $\phi_A(x)$ be the $\L'$-definition of $A$ and let $r_A$ be the satisfaction radius of $\phi_A$.
    \item Given a number $n$, enumerate guesses for the sequence $\lit{-r_A}, \ldots, \lit{n + r_A}$.
    \item For each such guess, use the satisfaction algorithm to compute a guess for the sequence $\llbracket A(\lit{0})\rrbracket^M, \ldots, \llbracket A(\lit{n})\rrbracket^M$. 
\end{enumerate}
Note that if the guess from the second step is correct then the guess from the last step will be too because in this case we have correctly identified $\lit{0}, \ldots, \lit{n}$, along with the $r_A$-neighborhood of each one.

There is only one problem with this algorithm: we may enumerate too many guesses. Suppose that our algorithm for enumerating guesses for the successor of an element of $M$ enumerates $k$ guesses. Then it seems that we might end up enumerating up to $k^n$ guesses for $a + \lit{n}$: $k$ guesses for $a + \lit{1}$, $k^2$ guesses for $a + \lit{2}$ (since each guess for $a + \lit{1}$ gives rise to $k$ guesses for $a + \lit{2}$), and so on. Thus in the algorithm above, we might end up with about $k^n$ guesses for the sequence $\llbracket A(\lit{0})\rrbracket^M, \ldots, \llbracket A(\lit{n})\rrbracket^M$, which is not enough to show that the sequence $\llbracket A(\lit{0})\rrbracket^M, \llbracket A(\lit{1})\rrbracket^M, \ldots $ is guessable.

The second key idea of our algorithm is that we actually don't end up with so many guesses. It is possible to show that since $\psi_S$ is mutually algebraic, if $M \models \psi_S(a, b)$ then---almost always---$a$ and $b$ are close to each other. In particular, if the radius of $\psi_S$ is $r$ then with only finitely many exceptions, $a$ and $b$ must be distance at most $r$ apart (this will be proved in Lemma~\ref{lemma:ma_close} below). If we ignore the finitely many exceptions, then this implies that for any $a$, every candidate for $S(a)$ is within distance $r$ of $a$. By induction, this implies that every candidate for $a + \lit{n}$ is within distance $rn$ of $a$. The point is that this means there are at most $rn$ such candidates (rather than $k^n$).

This does not quite solve our problem: even if there are only $rn$ candidates for $a + \lit{n}$, there could still be exponentially many candidates for the sequence $a - \lit{n}, \ldots, a + \lit{n}$. However, it can be combined with other tricks to reduce the number of guesses to $O(n^2)$. This will be explained in detail in the proof of Lemma~\ref{lemma:alg2}.

\medskip\noindent\textbf{Details of the algorithm.}
We will now describe the details of the algorithm and verify that it works correctly. We will break the algorithm (and its verification) into three parts, which work as follows.
\begin{enumerate}
    \item \textbf{The successor and predecessor guessing algorithm:} an algorithm which takes as input an element $a \in M$ and uses the existential formula approximating the successor relation to enumerate candidates for the successor and predecessor of $a$. This is described in Lemma~\ref{lemma:alg1}.
    \item \textbf{The neighborhood guessing algorithm:} an algorithm which takes as input an element $a \in M$ and a number $n$ and uses the ideas discussed above to enumerate candidates for the sequence $a - \lit{n},\ldots,a + \lit{n}$. This is described in Lemma~\ref{lemma:alg2}.
    \item \textbf{The \boldmath$A$ guessing algorithm:} an algorithm which takes as input a number $n$ and uses the neighborhood guessing algorithm together with the satisfaction algorithm to enumerate candidates for the sequence $\llbracket A(\lit{0})\rrbracket^M, \ldots, \llbracket A(\lit{n})\rrbracket^M$. This is described in Lemma~\ref{lemma:alg3}.
\end{enumerate}
Before describing these algorithms and proving their correctness, we need to prove one technical lemma (which is related to our comment above stating that if $M \models \psi_S(a, b)$ then $a$ and $b$ are usually close together).

\begin{lemma}
\label{lemma:ma_close}
Suppose that $\phi(x, y)$ is a formula (possibly with parameters from $M$) of radius $r$ which is mutually algebraic over $M$. There is a finite set $X$ of elements of $M$ such that if $M \models \phi(a, b)$ then either $a$ and $b$ are distance at most $r$ apart or at least one of $a, b$ is in $X$.\footnote{Note that if $T$ is modified in the way described in Remark~\ref{remark:generic} then both the statement and proof of this lemma can be simplified somewhat. In particular, we can replace the set $X$ with the $r$-neighborhood of $0$.}
\end{lemma}

\begin{proof}
It will help to first make explicit the parameters of $\phi$. Let $\bar{c}$ denote the tuple of parameters and write $\phi'(x, y, \bar{z})$ to denote the version of $\phi$ with the parameters exposed, i.e.\ $\phi(x, y)$ is $\phi'(x, y, \bar{c})$.

Call a pair $(a, b)$ \term{exceptional} if $a$ and $b$ are more than distance $r$ apart and both are more than distance $r$ from every coordinate of $\bar{c}$ and $M \models \phi(a, b)$. We will show that if $(a, b)$ is exceptional then the $r$-neighborhood type of $a$ occurs only finitely often in $M$, and likewise for $b$. Since there are only finitely many $r$-neighborhood types, this shows that there are only finitely many exceptional pairs. This is sufficient to finish the proof since we can take $X$ to consist of all elements which are part of some exceptional pair, together with the $r$-neighborhood of each coordinate of $\bar{c}$.

The claim about exceptional pairs follows from the indiscernability principle. Suppose $(a, b)$ is an exceptional pair. If $a'$ is any element of $M$ which is distance more than $r$ from all of $b$ and from every coordinate of $\bar{c}$ and which has the same $r$-neighborhood type as $a$ then by the indiscernability principle we have
\[
    M \models \phi(a, b) \implies M \models \phi'(a, b, \bar{c}) \implies M \models \phi'(a', b, \bar{c})
\]
and hence $M \models \phi(a', b)$. Since $\phi$ is mutually algebraic, there can only be finitely many such $a'$. Thus, outside of the $r$-neighborhood of $b$ and of each coordinate of $\bar{c}$, there are only finitely many elements with the same $r$-neighborhood type as $a$. Since these $r$-neighborhoods are themselves finite, they also contain only finitely many elements with the same $r$-neighborhood type as $a$ and thus we have shown that the $r$-neighborhood type of $a$ only occurs finitely often in $M$. Symmetric reasoning establishes the same result for $b$.
\end{proof}

\begin{lemma}[Guessing algorithm for successors and predecessors]
\label{lemma:alg1}
There is an algorithm which, given any $a \in M$ enumerates two lists of elements of $M$ such that
\begin{enumerate}
    \item $S(a)$ is in the first list and $P(a)$ is in the second list.
    \item There is a constant upper bound (independent of $a$) on the distance between any enumerated element and $a$.
\end{enumerate}
\end{lemma}

\begin{proof}
Let $\phi_S(x, y)$ be the $\L'$-definition of $S$ (i.e.\ $M \models S(a) = b$ if and only if $M \models \phi_S(a, b)$). Since $\phi_S(x, y)$ is mutually algebraic, we can apply Theorem~\ref{thm:qe1} to obtain a mutually algebraic existential $\L'$-formula $\psi_S(x, y)$ (which may contain parameters from $M$) such that $M \models \phi_S(x, y) \to \psi_S(x, y)$. Let $r$ be the radius of $\psi_S$. By Lemma~\ref{lemma:ma_close}, there is a finite set $X$ such that if $M \models \psi_S(b, c)$ then either $b$ and $c$ are distance at most $r$ apart or at least one of $b, c$ is in $X$. We will hard-code into our algorithm the elements of $X$, along with the identity of their successors and predecessors.

Note that since $\psi_S(x, y)$ is an existential formula, it follows that for a fixed $a$, the set of elements $b$ such that $M \models \psi_S(a, b)$ is computably enumerable (to see why, note that we can simply enumerate tuples in $M$ until we find one that witnesses the existential formula $\psi_S(a, b)$), and likewise for the set of elements $b$ such that $M \models \psi_S(b, a)$. Thus our algorithm may work as follows.
\begin{enumerate}
    \item Begin enumerating elements $b$ such that $M \models \psi_S(a, b)$ or $M \models \psi_S(b, a)$. 
    \item For each element $b$ such that $M \models \psi_S(a, b)$, check if either $a$ or $b$ is in $X$. If so, use the hard-coded list of successors and predecessors of elements of $X$ to check if $b$ is a successor of $a$. If this is true, enumerate $b$ into the first list. If $a$ and $b$ are both not in $X$ then enumerate $b$ into the first list with no extra checks.
    \item Do the same thing for each element $b$ such that $M \models \psi_S(b, a)$, but enumerate $b$ into the second list instead of the first.
\end{enumerate}
Since $M \models \phi_S(x, y) \to \psi_S(x, y)$, the true successor and predecessor of $a$ will be successfully enumerated. Also, if $b$ is some element of $M$ which is distance more than $r$ from $a$ then either $M \notmodels \psi_S(a, b)$ and $M \notmodels \psi_S(b, a)$, in which case $b$ will not be enumerated, or one of $a, b$ is in $X$, in which case $b$ will still not be enumerated (because it is not a true successor or predecessor of $a$). 
\end{proof}

\begin{lemma}[Guessing algorithm for neighborhoods]
\label{lemma:alg2}
There is an algorithm which, given any $a \in M$ and number $n \in \N$, enumerates a list of at most $O(n^2)$ guesses for the sequence $a - \lit{n}, \ldots, a + \lit{n}$, one of which is correct.
\end{lemma}

\begin{proof}
It is easiest to describe our algorithm in the following way. We will first describe an algorithm which has access to certain extra information (which might not be computable from an oracle for $M$) and which uses this extra information to correctly compute the sequence $a - \lit{n}, \ldots, a + \lit{n}$. We then obtain an algorithm for enumerating guesses for the sequence by trying each possible value of the extra information and running the algorithm on each of these values in parallel.\footnote{A slightly subtle point here is that the algorithm which uses extra information to compute $a - \lit{n}, \ldots, a + \lit{n}$ might not terminate if the extra information it is given is incorrect. Thus some of the possible values that we try for the extra information will never actually output a guess. This is why we only say that our final algorithm \emph{enumerates} a list of guesses rather than that it \emph{computes} a list of guesses.} To finish, we will have to show that there are only $O(n^2)$ possible values for the extra information. 

To begin, let $r_1$ be the constant from the statement of Lemma~\ref{lemma:alg1} (i.e.\ the upper bound on the distance between any $a$ and any element which is enumerated by the algorithm for guessing successors and predecessors of $a$). Let $\phi_S(x, y)$ be the $\L'$-definition of $S$ and let $r_2$ be the satisfaction radius of $\phi_S$.

Suppose we are given an element $a \in M$ and a number $n \in \N$ as input. Let $N = r_1n + r_2$. Our algorithm proceeds in two phases.
\begin{enumerate}
    \item In the first phase, we will use the algorithm from Lemma~\ref{lemma:alg1} to collect candidates for $a + \lit{i}$ for each $-N \leq i \leq N$. More precisely, for each such $i$ we will find a set $U_i$ which contains $a + \lit{i}$ and which is contained in the $r_1|i|$-neighborhood of $a$.
    \item In the second phase, we will use the sets of candidates collected in the first stage as input to the satisfaction algorithm (applied to $\phi_S$) to determine the exact identities of $a + \lit{i}$ for each $-n \leq i \leq n$.
\end{enumerate}

The ``extra information'' that we alluded to above is needed in the first phase of the algorithm. This is because the sets $U_i$ are not quite computable from an oracle for $M$, but only computably enumerable. However, since the they are all finite, it is possible to compute them exactly with only a small amount of additional information. Let $i$ be the index of the last $U_i$ to have a new element enumerated into it and let $m$ be the size of $U_i$ once all its elements have been enumerated (note that such an $i$ and $m$ exist because all the $U_j$ are finite). We claim that the pair $(i, m)$ is enough information to allow us to compute all the sets $U_j$ exactly and that there are only $O(n^2)$ possible values for this pair.

To see why we can compute all the $U_j$ exactly, note that given $i$ and $m$ we can simply keep enumerating elements into all the $U_j$ until we see that $U_i$ has size $m$. To see why there are only $O(n^2)$ possible values for the pair $(i, m)$, note that there are only $2N + 1$ possible values for $i$ and at most $r_1(2N + 1)$ possible values for $m$ (since $U_i$ is contained in the $r_1|i|$-neighborhood of $a$, which has $r_1(2|i| + 1) \leq r_1(2N + 1)$ elements). Thus there are at most $r_1(2N + 1)^2 = O(n^2)$ possible values for $(i, m)$.

\medskip\noindent\textit{Phase 1: collecting candidates.}
The sets $U_i$ for $-N \leq i \leq N$ can be enumerated as follows. To begin with, set $U_0 = \{a\}$ and set all other $U_i = \0$. Then run the following processes in parallel: for each $-N  < i < N$ and each element $b$ of $U_i$, use the algorithm of Lemma~\ref{lemma:alg1} to enumerate candidates for the successor and predecessor of $b$. If $i \geq 0$ then add each such candidate for the successor of $b$ to $U_{i + 1}$. If $i \leq 0$ then add each candidate for the predecessor of $b$ to $U_{i - 1}$. It is easy to show by induction that for each $i$, $a + \lit{i}$ will eventually be enumerated into $U_i$ and that each element enumerated into $U_i$ is distance at most $r_1|i|$ from $a$.

\medskip\noindent\textit{Phase 2: computing neighbors exactly.}
Given the sets $U_i$ from phase 1, we can compute the exact identities of $a - \lit{n}, \ldots, a + \lit{n}$ as follows. First, let $U = U_{-N} \cup \ldots \cup U_N$ and note that $a + \lit{0} = a$. Next, loop over $i = 0, 1, \ldots, n - 1$. On step $i$, we will compute $a + \lit{i + 1}$ and $a - \lit{(i + 1)}$. Suppose that we are on step $i$ of the algorithm and assume for induction that we have already successfully computed $a + \lit{i}$ and $a - \lit{i}$ (note that for $i = 0$ this is trivial). Now do the following:
\begin{enumerate}
    \item For each $b \in U_{i + 1}$, use the satisfaction algorithm (of Proposition~\ref{prop:satisfaction}) with the set $U$ to check if $M \models \phi_S(a + \lit{i}, b)$. 
    \item For each $b \in U_{-(i + 1)}$, use the satisfaction algorithm with the set $U$ to check if $M \models \phi_S(b, a - \lit{i})$.
\end{enumerate}
Note that each $b \in U_{i + 1}$ is within distance $r_1(i + 1)$ of $a$. Since $U$ contains the entire $N$-neighborhood of $a$ and $N = r_1n + r_2 \geq r_1 (i + 1) + r_2$, $U$ also contains the $r_2$-neighborhood of $b$. Thus the conditions of the satisfaction algorithm are fulfilled and so we correctly compute whether $b$ is the successor of $a + \lit{i}$ or not. And since $U_{i + 1}$ is guaranteed to contain $a + \lit{i + 1}$, our algorithm will correctly identify $a + \lit{i + 1}$. Completely symmetric reasoning applies to show that our algorithm will correctly identify $a - \lit{(i + 1)}$.
\end{proof}

\begin{lemma}[Guessing algorithm for $A$]
\label{lemma:alg3}
There is an algorithm which, given any number $n \in \N$, enumerates a list of at most $O(n^2)$ guesses for the sequence $\llbracket A(\lit{0})\rrbracket^M, \ldots, \llbracket A(\lit{n})\rrbracket^M$, one of which is correct.
\end{lemma}

\begin{proof}
Let $\phi_A(x)$ be the $\L'$-definition of $A$ (i.e.\ $M \models A(a)$ if and only if $M \models \phi_A(x)$) and let $r$ be the satisfaction radius of $\phi_A$. This algorithm essentially just combines the algorithm for guessing neighborhoods with the satisfaction algorithm for $\phi_A$.

Given a number $n \in \N$ as input, first use the algorithm from Lemma~\ref{lemma:alg2} to enumerate guesses for the sequence $-\lit{(n + r)}, \ldots, \lit{n + r}$ (this can be done by simply giving the element $0 \in M$ and the number $n + r$ as input to that algorithm). Let $b_{-(n + r)}, \ldots, b_{n + r}$ be one such guess and let $U = \{b_i \mid -(n + r) \leq i \leq n + r\}$. For each $0 \leq i \leq n$, use the satisfaction algorithm with the set $U$ to check if $M \models \phi_A(b_i)$. If so, report that $\tv{A(\lit{i})}{M}$ is true and otherwise report that it is false.

So for each guess for the sequence $-\lit{(n + r)}, \ldots, \lit{n + r}$ we generate exactly one guess for the sequence $\llbracket A(\lit{0})\rrbracket^M, \ldots, \llbracket A(\lit{n})\rrbracket^M$ and thus we generate at most $O((n + r)^2) = O(n^2)$ guesses overall. Furthermore, one of the guesses for the sequence $-\lit{(n + r)}, \ldots, \lit{n + r}$ is guaranteed to be correct. For this guess, each $b_i$ is actually equal to $\lit{i}$ and for each $i \leq n$, the set $U$ really does contain the $r$-neighborhood of $b_i$. Thus, for this guess, each $\llbracket A(\lit{i}) \rrbracket^M$ is computed correctly.
\end{proof}

\section{Questions}

\subsection{Bi-interpretability}

Since definitional equivalence is a strong form of bi-interpretability, it seems reasonable to ask whether Theorem~\ref{thm:main} still holds when definitional equivalence is replaced with bi-interpretability. 

\begin{question}
Is there a consistent, c.e.\ theory such that no theory bi-interpretable with it has a computable model?
\end{question}

It seems possible that the theory $T$ we used in our proof of Theorem~\ref{thm:main} could also be used to answer this question, but there are a few difficulties. One issue is that while the mutual algebraicity of a structure is preserved under definitional equivalence, it is not always preserved under bi-interpretability. 

\begin{example}[Bi-interpretability fails to preserve mutual algebraicity]
Let $\L$ be the language with just equality and let $\L'$ be a language with two sorts $U$ and $V$, and two function symbols $f, g \colon V \to U$. Let $T$ be the $\L$-theory describing an infinite set and let $T'$ be the $\L'$-theory which states that $U$ is infinite and $(f, g)$ is a bijection from $V$ to $(U\times U) \setminus \{(x, x) \mid x \in U\}$.

Given a model of $T'$, we can obtain a model of $T$ by forgetting the sort $V$ and the functions $f$ and $g$. Given a model of $T$ we can obtain a model of $T'$ as follows. Take as the underlying set for the model, the set of all pairs $(x, y)$ with pairs of the form $(x, x)$ forming the sort $U$ and pairs of the form $(x, y)$ for $x\neq y$ forming the sort $V$. For the functions $f$ and $g$, simply take $f((x, y)) = (x, x)$ and $g((x, y)) = (y, y)$. 

It is not hard to check that these two interpretations give a bi-interpretation. However, while every model of $T$ is clearly mutually algebraic, the same is not true for $T'$. For example, the formula $f(y) = x$ is not equivalent to any Boolean combination of mutually algebraic formulas.
\end{example}

A second issue (not unrelated to the first) is that, in our proof, we relied on the fact that any model $M$ of a theory definitionally equivalent to $T$ carries a notion of distance inherited from $T$. In particular, we used this to bound the number of guesses required by the neighborhood guessing algorithm of Lemma~\ref{lemma:alg2}. However, if $M$ is only a model of a theory bi-interpretable with $T$, it is not clear if there is still a good notion of distance which can play this role.

\subsection{Natural theories}
Arguably, the theory $T$ that we used to prove Theorem~\ref{thm:main} is not very natural. It would be interesting to know if this is necessary.

\begin{question}
Is there a natural theory witnessing Theorem~\ref{thm:main}?
\end{question}

Of course, much depends on what the word ``natural'' means. In the interests of asking a somewhat more concrete question, let's say that a theory is natural if it has been studied (at least implicitly) by mathematicians who are not logicians.

We can rephrase our question as follows: is there any natural theory which has satisfies the robust version of the Tennenbaum property implicit in Theorem~\ref{thm:main}? In light of Pakhomov's results, which seem to show that any theory interpreting a decent amount of arithmetic is definitionally equivalent to a theory without the Tennenbaum property, it seems like a good idea to first ask whether any natural theory satisfies the regular version of the Tennenbaum property but does not interpret any nontrivial fragment of arithmetic. We are not aware of any such theory and would consider it highly interesting.

\begin{question}
Is there any natural (consistent) theory $T$ such that $T$ has no computable models and does not interpret any nontrivial fragment of arithmetic?
\end{question}

One can ask a similar question on the level of models rather than theories. In analogy with our definition for theories, let's say that a countable structure is natural if it has been studied by mathematicians who are not logicians. 

\begin{question}
Is there a natural countable structure with no computable presentation?
\end{question}

Again, we are not aware of any completely convincing example of such a structure and would consider any such example to be very interesting.

\bibliographystyle{alpha}
\bibliography{bibliography}

\end{document}